\documentclass[10pt,a4paper]{amsart}
\usepackage{etex}
\usepackage{amssymb}
\usepackage{amsthm}
\usepackage[autostyle]{csquotes}
\usepackage{amsfonts}
\usepackage{microtype}
\usepackage{mathrsfs}
\usepackage{graphicx}
\usepackage{enumerate}
\usepackage{color}
\usepackage{latexsym}
\usepackage{rotating}
\usepackage{xspace}
\usepackage[all]{xy}
\usepackage{longtable}
\usepackage{leftidx}
\usepackage{mathtools}
\usepackage{titlesec}
\usepackage{tikz}
\usetikzlibrary{arrows}
\usetikzlibrary{matrix,arrows,decorations.pathmorphing}
\usepackage{geometry}

\newtheorem{theorem}{Theorem}[section]
\numberwithin{equation}{theorem}
\newtheorem{lemma}[theorem]{Lemma}

\newtheorem{corollary}[theorem]{Corollary}

\theoremstyle{definition}
\newtheorem{definition}[theorem]{Definition}
\newtheorem{example}[theorem]{Example}

\theoremstyle{conjecture}

\newcommand{\Ass}{\operatorname{Ass}}

\newcommand{\Spec}{\operatorname{Spec}}

\newcommand{\ara}{\operatorname{ara}}

\newcommand{\id}{\operatorname{id}}
\newcommand{\fd}{\operatorname{fd}}

\newcommand{\pd}{\operatorname{pd}}

\newcommand{\Ext}{\operatorname{Ext}}
\newcommand{\Supp}{\operatorname{Supp}}

\newcommand{\Tor}{\operatorname{Tor}}
\newcommand{\Hom}{\operatorname{Hom}}
\newcommand{\Att}{\operatorname{Att}}
\newcommand{\Ann}{\operatorname{Ann}}
\newcommand{\Rad}{\operatorname{Rad}}

\newcommand{\depth}{\operatorname{depth}}

\newcommand{\Coass}{\operatorname{Coass}}
\newcommand{\coker}{\operatorname{coker}}

\newcommand{\Max}{\operatorname{Max}}

\newcommand{\lo}{\longrightarrow}
\newcommand{\fm}{\frak{m}}
\newcommand{\fp}{\frak{p}}
\newcommand{\fq}{\frak{q}}
\newcommand{\fr}{\frak{r}}
\newcommand{\fa}{\frak{a}}
\newcommand{\fb}{\frak{b}}

\newcommand{\fn}{\frak{n}}

\newcommand{\E}{\mbox{E}}
\newcommand{\D}{\mbox{D}}

\newcommand{\F}{\mbox{F}}

\renewcommand{\H}{\mbox{H}}

\newcommand{\suchthat}{\;\ifnum\currentgrouptype=16 \middle\fi|\;}

\makeatletter
\newcommand{\holim@}[2]{%
\vtop{\m@th\ialign{##\cr
\hfil$#1\operator@font holim$\hfil\cr
\noalign{\nointerlineskip\kern1.5\ex@}#2\cr
\noalign{\nointerlineskip\kern-\ex@}\cr}}%
}
\newcommand{\holim}{%
\mathop{\mathpalette\holim@{\rightarrowfill@\textstyle}}\nmlimits@
}
\makeatother

\makeatletter
\def\@secnumfont{\bfseries}
\def\section{\@startsection{section}{1}%
\z@{.7\linespacing\@plus\linespacing}{.5\linespacing}%
{\normalfont\Large\bfseries\filcenter}}
\def\subsection{\@startsection{subsection}{2}%
\z@{.5\linespacing\@plus.7\linespacing}{-.5em}%
{\normalfont\large\bfseries}}
\makeatother
\begin{document}
	
\author[K. Divaani-Aazar and M. Eghbalii]
{Kamran Divaani-Aazar and Mohammad Eghbali}
	
\title[Regularity, local cohomology and ...]
{Regularity, local cohomology and injective modules under contracting endomorphisms}

\address{K. Divaani-Aazar, Department of Mathematics, Faculty of Mathematical Sciences, Alzahra University,
Tehran, Iran.}
\email{kdivaani@ipm.ir}

\address{M. Eghbali, Department of Mathematics, Faculty of Mathematical Sciences, Alzahra University, Tehran,
Iran.}
\email{m.eghbali007@yahoo.com}

	
	
\begin{abstract}
This work extends several fundamental results previously established for rings of characteristic
$p$ to the broader class of rings that admit a contracting endomorphism. Specifically, let $\phi:
R \to R$ be such an endomorphism, and denote by $^{\phi}R$ the ring $R$ endowed with the left
$R$-module structure induced by $\phi$. We employ the functor $\Hom_R(^{\phi}R,-)$ to derive new
characterizations of regular rings. Furthermore, we examine the behavior of the functor $(-)
\otimes_R {^{\phi}}R$ on local cohomology modules and injective modules.
\end{abstract}

\maketitle

\tableofcontents
	
\section{Introduction}

The study of modules over rings of prime characteristic has long been a topic of interest among algebraists.
Let $R$ be a ring of prime characteristic $p$. The ring endomorphism $f:R\rightarrow R$ given by $f(x)=x^p$
is called the {\it Frobenius endomorphism}. The ring $R$ naturally acquires the structure of a left $R$-module
via $f$, which is denoted by $^f\hspace{-0.5mm}R$. The endofunctor $\F(-)=(-)\otimes_R {^{f}\hspace{-0.5mm}R}$
on the category of $R$-modules is also known as the Frobenius functor.

Kunz \cite{Kun69} was a pioneer in this field. His well-known result states that a local ring
$(R,\mathfrak{m},\Bbbk)$ of prime characteristic is regular if and only if the Frobenius functor is exact.
This fundamental theorem has been the basis for many generalizations. Furthermore, the Frobenius functor plays
a significant role in studying local cohomology modules.

When $(R,\mathfrak{m},\Bbbk)$ is a local ring, a ring endomorphism $\phi: R\to R$ is called {\it contracting}
if there exists a positive integer $i$ such that $\phi^i(\mathfrak{m})\subseteq \mathfrak{m}^2$. Also, a ring
endomorphism $\phi :R \to  R$ is called \emph{locally contracting} if the map $\phi_{\fp} :R_{\fp}\to R_{\fp}$
defined by $\phi_{\fp}(r/s)=\phi(r)/\phi(s)$ is a contracting endomorphism for every prime ideal $\fp$ of
$R$. It is routine to verify that the Frobenius endomorphism is locally contracting. Recently, the study of
contracting endomorphisms has attracted considerable attention. Avramov, Iyengar, and Miller \cite{AIM06}
investigated the structure of modules over contracting endomorphisms. Their work demonstrated that some favorable
properties of the Frobenius endomorphism could extend to general contracting endomorphisms. See also \cite{EY12,
FM21, Rah09}. In this work, we continue this line of inquiry and derive several useful and well-known properties
of the Frobenius endomorphism for general contracting endomorphisms.

Let $\phi:R\to R$ be a ring endomorphism. We equip $R$ with a left $R$-module structure via $\phi$, denoted by
${}^{\phi}\!R$. If the $R$-module ${}^{\phi}\!R$ is finitely generated, then $R$ is said to be $\F^{\phi}\!$-finite.
We consider two endofunctors: $\F^{\phi}_R(-)=(-)\otimes_R {}^{\phi}\!R$ and $\widetilde{\F}^{\phi}_R(-)=
\Hom_R({}^{\phi}\!R,-)$ on the category of $R$-modules. Our study focuses on the properties of these functors when
$\phi$ is contracting or locally contracting.

In Section 2, we gather some necessary preliminaries regarding contracting and locally contracting endomorphisms.

In Section 3, we present two characterizations of regular rings using the functor $\widetilde{\F}^{\phi}(-)$, which
are detailed in Theorems \ref{4.1a} and \ref{4.4a}. Theorem \ref{4.1a} shows that for a local ring $(R,\fm,\Bbbk)$
admitting a contracting endomorphism  $\phi:R\to R$, the conditions of $R$ being regular and $R/{\phi(\fm)R}$ being
Artinian are equivalent to the exactness of the endofunctor $\widetilde{\F}^{\phi}(-)$ on the category of Artinian
$R$-modules. Theorem \ref{4.1a} has an immediate consequence for Frobenius endofunctor, which can be viewed as a
dual to Kunz’s theorem;  see Corollary \ref{4.2a}. Theorem \ref{4.4a} states that a local ring $(R,\fm,\Bbbk)$
admitting a contracting endomorphism $\phi:R\to R$ such that $R$ is $\F^{\phi}\!$-finite is regular if and only if
the functors $\widetilde{\F}^{\phi}_R(-)$ and $\F^{\phi}_R(-)$ are naturally isomorphic on the category of finitely
generated $R$-modules.

In Section 4, we explore the interplay between the functor $\F^{\phi}_R(-)$ and local cohomology modules. Let
$R$ be a regular ring and $\phi:R \to R$ a locally contracting endomorphism such that $R$ is $\F^{\phi}\!$-finite.
We show that if for some ideal $\fa$ of $R$ and a nonnegative integer $\ell$, the $R$-module $\H^{\ell}_{\fa}(R)$
is Artinian, then it is injective; see Theorem \ref{5.4a}.

Finally, Section 5 offers an improvement over some of Marley's results in \cite{MA14}. This section investigates
the preservation of injective modules under the functor $\F^{\phi}(-)$. As the main result of this section, we prove
that if $R$ is a quasi-Gorenstein ring and $\phi:R \to R$ is a locally contracting endomorphism, then $\F^{\phi}(I)
\cong I$ for every injective $R$-module $I$; see Theorem \ref{3.6a}. We conclude this section by considering the
behavior of associated and coassociated prime ideals under the action of the functors $\F^{\phi}(-)$ and
$\widetilde{\F}^{\phi}(-)$, as detailed in Corollaries \ref{3.7a} and \ref{3.8a}.

\section{Preliminaries}

Throughout this paper, $R$ denotes a commutative Noetherian ring with identity, and $\phi :R\to  R$ a ring endomorphism.
For a positive integer $e$, let $\phi^e = \phi \circ\dots\circ \phi:R \to R$ denotes the $e$-th iteration of $\phi$. In
particular, $\phi^1=\phi$. Given an $R$-module $M$, we write $^{{\phi}^e}\hspace{-0.5mm} M$ for the Abelian group $M$
endowed with a left $R$-module structure via ${\phi}^e$,  where the action is given by $r.m= {\phi}^e(r)m$ for $r\in R$
and $m\in M$.

For each positive integer $e$, let $\F^{\phi^e}_R(-)$ denote the endofunctor $-\otimes_R {^{\phi^e}}R$ on the category
of $R$-modules. The right $R$-module structure on $\F^{\phi^e}_R(M)$ is given by $(x\otimes s).r=x\otimes (sr)$ for
$x\in M$, $s \in {^{{\phi}^e}}\hspace{-0.5mm}R$ and $r\in R$. Note that $(rx) \otimes s=x \otimes (\phi^e(r)s)$. When the
ring $R$ is clear from the context, we may write $\F^{\phi^e}(-)$ instead of $\F^{\phi^e}_R(-)$. The functor $\F^{\phi^e}$
is an additive, right exact functor that commutes with direct sums and direct limits.

Dual to the functor $\F^{\phi^e}_R(-)$, we consider the endofunctor $\widetilde{\F}^{\phi^e}_R(-)=\Hom_R({^{\phi^e}}R,-)$
on the category of $R$-modules. For an $R$-module $M$, $\widetilde{\F}^{\phi^e}_R(M)$ is an Abelian group with the right
$R$-module structure given by $(\theta .r)(s)=\theta (sr)$ for $r\in R$, $\theta\in \Hom_R ({^{\phi^e}\hspace{-0.5mm}R},M)$
and $s\in {^{{\phi}^e}}\hspace{-0.5mm}R$. Since $\theta$ is a homomorphism of left $R$-modules, we have $r\theta(s)=\theta
(\phi^e(r) s)$. The functor $\widetilde{\F}^{\phi^e}_R(-)$ is an additive, left exact functor that commutes with direct
products and inverse limits. When the ring $R$ is clear from the context, we may write $\widetilde{\F}^{\phi^e}(-)$ instead
of $\widetilde{\F}^{\phi^e}_R(-)$.

Note that if $R$ has positive prime characteristic $p$, we can consider $\phi$  to be the Frobenius endomorphism
$f(x)=x^p$ and $\F^{\phi}$ to be the Frobenius functor $\F(-)=(-)\otimes_R {^{f}\hspace{-0.5mm}R}$.

We begin with the following definitions from \cite[12.1]{AIM06} and \cite[Definition 3.3]{EY12}.

\begin{definition}\label{2.1}
\begin{itemize}
\item[(i)] Let $(R,\fm,\Bbbk)$ be a local ring.  A ring endomorphism $\phi :R \to  R$ is called {\it contracting}
if there exists a positive integer $i$ such that $\phi ^i(\fm)\subseteq {\fm}^2$.
\item[(ii)] A ring endomorphism $\phi:R \to  R$ is called {\it locally contracting} if for every prime ideal $\fp$
of $R$, the induced map $\phi_{\fp} :R_{\fp}\to R_{\fp}$ defined by $\phi_{\fp}(r/s)=\phi(r)/\phi(s)$ is a
contracting endomorphism.
\end{itemize}
\end{definition}	
	
\begin{definition}\label{2.2} Let $\phi :R\to  R$ be a ring endomorphism. The ring $R$ is said to be
$\F^{\phi}$-{\it finite} if the $R$-module $^{{\phi}}\hspace{-0.5mm} R$ is finitely generated.
\end{definition}

We summarize several useful properties of contracting and locally contracting endomorphisms in the following result.
In what follows, for an ideal $\fa$ of $R$, the radical of $\fa$ is denoted by $\Rad(\fa)$.
	
\begin{lemma}\label{2.3} Let $\phi:R \to R$ be a ring endomorphism, $M$ an $R$-module and $e$ a positive integer.
The following properties hold:
\begin{itemize}
\item[(i)]  $\F^{\phi^e}(R) \cong R$, as right $R$-modules.
\item[(ii)]  If $\alpha: M \to M$ is multiplication by $x\in R$, then $\F^{\phi^e}(\alpha)$ is multiplication by
$\phi^e(x)$.
\item[(iii)]  Let $\beta:R^n \to R^m$ be an $R$-homomorphism represented by a matrix $(a_{ij})$. Then $\F^{\phi^e}(\beta)$
is represented by the matrix $(\phi^e (a_{ij}))$.
\item[(iv)]  For any ideal $\fa=\langle a_1, \dots, a_t\rangle$ of $R$, we have $\F^{\phi^e} (R/\fa) \cong
R/{\phi^e(\fa)R}$, where $\phi^e(\fa)R=\langle\phi^e(a_1), \dots, \phi^e(a_t)\rangle$.
\item[(v)]  Let $T$ be an $R$-algebra with an algebra endomorphism $\psi: T \to T$. Then, there is an isomorphism of
$T$-modules: $\F^{\phi^e}_R(M)\otimes_R T \cong \F^{\psi^e}_T (M\otimes_R T)$.
\item[(vi)] Let $\phi$ be locally contracting. Then for any prime ideal $\fp$ of $R$, we have $\Rad(\phi(\fp)R)=\fp$.
In particular, for every maximal ideal $\fm$, the ideal $\phi^{e}(\fm)R$ is
$\fm$-primary.
\item[(vii)] If $R$ is local and $\phi$ is locally contracting, then it follows that $\phi$ is contracting. Furthermore,
if $R$ is local and $\phi$ is contracting, then the induced ring endomorphism $\widehat{\phi}:\widehat{R}\to \widehat{R}$
is likewise contracting.
\item[(viii)] Assume that $R$ is local with the maximal ideal $\fm$ and $\phi$ is locally contracting. If $R$ is
$\F^{\phi}\!$-finite, then there is an isomorphism of $\widehat{R}$-modules $\widehat{R}\otimes_R {}^{\phi}\!R\cong
{}^{\widehat{\phi}}\!\widehat{R}$, and so $\widehat{R}$ is $\F^{\widehat{\phi}}$-finite.
\item[(ix)] If $R$ is local with the maximal ideal $\fm$ and $\phi$ is contracting, then $\phi(\fm) \subseteq \fm$.
\item[(x)] If $R$ is $\F^{\phi}\!$-finite, then for every maximal ideal $\fm$ of $R$, the quotient ring $R/{\phi(\fm)R}$
is Artinian.
\item[(xi)] If $\phi:R \to R$ is contracting (respectively, locally contracting), then the ring endomorphism $\phi^{e}:R
\to R$ is also contracting (respectively, locally contracting). Furthermore, if $R$ is $\F^{\phi}\!$-finite, then it is
also $\F^{\phi^e}\!$-finite.
\end{itemize}
\end{lemma}

\begin{proof} (i) and (ii) are straightforward.

(iii), (iv), (v) The proofs of these statements are analogous to the arguments used for the Frobenius endomorphism in
characteristic p settings, which are well-documented in the literature; see, for example, \cite[Propositions 8.2.2,
8.2.4, and 8.2.5]{BH98}.

(vi) See \cite[Lemma 3.5]{EY12}.

(vii) is clear.

(viii) Since the $R$-module $^{\phi}R$ is finitely generated, we obtain an $\widehat{R}$-isomorphism $\widehat{R}\otimes_R \hspace{-0.5mm}^{\phi}R\cong \hspace{-0.5mm} \widehat{^\phi R}$. On the other hand, by (vi), the ideal $\phi(\fm)R$ is
$\fm$-primary,  which implies that the filtrations $\{(\phi(\fm)R)^n\}_n$ and $\{\fm^n\}_n$ are cofinal. Hence, they induce
the same topology on $R$, and so $\widehat{^\phi R}=\hspace{-0.5mm} ^{\widehat{\phi}}\widehat{R}$. Consequently, we get
the desired $\widehat{R}$-isomorphism $\widehat{R}\otimes_R \hspace{-0.5mm}^{\phi}R\cong \hspace{-0.5mm}
^{\widehat{\phi}}\widehat{R}$.

(ix) Since $\phi$ is contracting, there exists a positive integer $i$ such that $\phi^i(\fm)\subseteq \fm^2$. From this,
it follows that $\phi(\fm)R \neq R$, and hence $\phi(\fm)\subseteq \fm$.

(x) Let $\fm$ be a maximal ideal of R. Choose a prime ideal $\fp$ of $R$ such that $\phi(\fm)R\subseteq \fp$.
Then $\fm \subseteq \phi^{-1}(\fp)$, and thus $\fm=\phi^{-1}(\fp)$. As $R$ is $\F^{\phi}\!$-finite, it follows
that $\fp$ is a maximal ideal of $R$. Consequently, the quotient ring $R/{\phi(\fm)R}$ is Artinian.
	
(xi) is clear.
\end{proof}

Among other things, the following example exhibits a contracting endomorphism that is not locally contracting.

\begin{example}\label{2.31} Let $(R,\fm,\Bbbk)$ be a local ring and $\phi:R \to R$ a ring endomorphism.
By Lemma \ref{2.3}(vi) and (x), if $\phi$ is locally contracting or $R$ is $\F^{\phi}\!$-finite, then the
ideal $\phi(\fm)R$ is $\fm$-primary. One might ask whether this remains true when $\phi$ is merely
contracting; however, it does not.  Let $\Bbbk$ be a field, $R=\Bbbk[[X,Y]]$ and $\fm=\langle X, Y \rangle$.
The map $\phi:R \to R$ defined by $\phi(f(X,Y))=f(X^2,0)$ is contracting, yet $\phi(\fm)R=\langle X^2 \rangle$
is not $\fm$-primary. Also, Lemma \ref{2.3}(vi) shows that $\phi$ is not locally contracting.
\end{example}

We conclude this section by recalling the following characterizations of regular rings. For the first one,
refer to \cite[Theorem 13.3]{AIM06}.

\begin{lemma}\label{2.4} Let $(R,\fm,\Bbbk)$ be a local ring with a contracting endomorphism $\phi:R \to R$.
The following are equivalent:
\begin{itemize}
\item[(i)] $R$ is regular.
\item[(ii)] $\fd_R(^{\phi^e}R) = \dim R/{(\phi^e(\fm)R)}$ for every integer $e>0$.
\item[(iii)] $\fd_R(^{\phi^e}R)< \infty$ for some integer $e>0$.
\item[(iv)] $\id_R(^{\phi^e}R)= \dim R$ for every integer $e>0$.
\item[(v)] $\id_R(^{\phi^e}R) < \infty$ for some integer $e>0$.
\end{itemize}
\end{lemma}	

Epstein and Yongwei established a non-local analogue of the above result as follows; see \cite[Theorem 3.10]{BH98}.

\begin{lemma}\label{2.4aa} Let $R$ be a reduced ring with a locally contracting endomorphism $\phi:R \to R$.
The following are equivalent:
\begin{itemize}
\item[(i)] $R$ is regular.
\item[(ii)] ${}^{\phi^e}\!R$ is flat.
\item[(iii)] $\Ass_R(\F^{\phi^e}(M))=\Ass_R(M)$ for every $R$-module $M$ and every integer $e>0$.
\end{itemize}
\end{lemma}	

\section{Characterization of regular rings via the functor $\widetilde{\F}^{\phi}$}

In this section, Theorems \ref{4.1a} and \ref{4.4a} constitute the main results, providing characterizations of
regular local rings via contracting endomorphisms. We begin with the following useful result.

\begin{lemma}\label{3.1a}  Let $\psi:(R,\fm,\Bbbk)\to (S,\fn,\Bbb L)$ be a homomorphism of local rings such that
$S$ is finitely generated as an $R$-module via $\psi$. Then, for any $R$-module $M$, we have an $S$-isomorphism
$$\Hom_R\left(M,\E_R\left(\Bbbk\right)\right)\otimes_R S \cong \Hom_S\left(\Hom_R\left(S,M\right),\E_S\left(\Bbb
L\right)\right).$$
\end{lemma}
	
\begin{proof} By \cite[Lemma 3.7]{MA14}, there is an $S$-isomorphism$$\Hom_R\left(S,\E_R\left(\Bbbk \right)\right)
\cong \E_S\left(\Bbb L \right).$$ Hence, we obtain the following chain of $S$-isomorphisms:
\[\begin{array}{rlllllllllll}
\Hom_R\left(M,\E_R\left(\Bbbk \right)\right)\otimes_RS &\cong \Hom_R\left(\Hom_R\left(S,M\right),\E_R\left(\Bbbk
\right)\right)\\
&\cong\Hom_R\left(\Hom_R\left(S,M\right)\otimes_SS,\E_R\left(\Bbbk \right)\right) &  \\
&\cong \Hom_S\left(\Hom_R\left(S,M\right),\Hom_R\left(S,\E_R\left(\Bbbk \right)\right)\right)\\
&\cong \Hom_S\left(\Hom_R\left(S,M\right),\E_S\left(\Bbb L \right)\right).
\end{array}\]
The first isomorphism follows from the Hom evaluation map, while the third isomorphism results from the adjointness
between the $\Hom$ and tensor product functors.
\end{proof}

From now on, when $(R,\fm,\Bbbk)$ is a local ring, we let $\D(-)=\Hom_R(-,\E_R(\Bbbk))$ denote the Matlis duality
functor. As an immediate consequence of the above result, we have the first part of the following corollary.

\begin{corollary}\label{3.2a} Let $(R,\fm,\Bbbk)$ be a local ring with a ring endomorphism $\phi: R\rightarrow R$.
Assume that $R$ is $\F^{\phi}\!$-finite. Then, for any $R$-module $M$ and each positive integer $e$, there are the
following natural $R$-isomorphisms:
\begin{itemize}
\item[(i)] $D(\widetilde{\F}^{{\phi}^e}(M))\cong  \F^{{\phi}^e}(\D(M))$.
\item[(ii)] $\D(\F^{{\phi}^e}(M))\cong \widetilde{\F}^{{\phi}^e}(\D(M))$.
\end{itemize}
\end{corollary}

\begin{proof} (i) is immediate by Lemma \ref{3.1a}. Note that as $R$ is $\F^{\phi}\!$-finite, by Lemma \ref{2.3}(xi), it
is also $\F^{\phi^e}\!$-finite.

(ii) See \cite[Proposition 3.8]{MA14}.
\end{proof}

We begin by presenting our first characterization of regular rings using the functor $\widetilde{\F}^{\phi}$.
		
\begin{theorem}\label{4.1a} Let $(R,\fm,\Bbbk)$ be a local ring with a contracting endomorphism
$\phi:R\rightarrow R$ and $\mathbb{A}(R)$ denote the category of Artinian $R$-modules. The following are equivalent:
\begin{itemize}
\item[(i)] $R$ is regular and the quotient ring $R/{\phi(\fm)R}$ is Artinian.
\item[(ii)] The functor $\Hom_R ({^{\phi}}\hspace{-0.5mm}R,-)$ is exact on $\mathbb{A}(R)$.
\end{itemize}
\end{theorem}
	
\begin{proof} (i)$\Rightarrow$(ii) By Lemma \ref{2.4}, the $R$-module ${^{\phi}}\hspace{-0.5mm}R$ is flat. Consequently,
by \cite[Example on page 63]{Xu06},  $\Ext_R^1({^{\phi}}R,A)=0$ for every Artinian $R$-module $A$. Consider an exact
sequence of Artinian $R$-modules: $$0 \to A \to B \to C\to 0. \  \  (\dag)$$ Applying the functor $\Hom_R({^{\phi}}R,-)$
to $(\dag)$ yields the exact sequence: $$0\to \Hom_R({^{\phi}}R,A)\to  \Hom_R({^{\phi}}R,B)\to \Hom_R ({^{\phi}}R,C)\to
0.$$
Therefore, the functor $\Hom_R ({^{\phi}}\hspace{-0.5mm}R, -)$ is exact on $\mathbb{A}(R)$.
		
(ii)$\Rightarrow$(i) We begin by showing that the $R$-module $^{\phi}R$ is flat. This will be accomplished by showing
that $\Tor_1^R (M,{^{\phi}}R)=0$ for every finitely generated $R$-module $M$. Let $M$ be a finitely generated $R$-module.
We may consider an exact sequence: $$0 \to K \to R^n \to M \to 0.$$ Let $\E=\E_R(\Bbbk)$. Applying the exact functor
$\D(-)$ to this sequence
induces the exact sequence of Artinian $R$-modules: $$0\to \D(M)\to E^n\to \D(K)\to 0.$$ By applying the functor
$\Hom_R({^{\phi}}R,-)$ to the latter sequence, we obtain the exact sequence: $$0\to \Hom_R ({^{\phi}}R, \D(M))\to
\Hom_R ({^{\phi}}R, E^n) \to \Hom_R ({^{\phi}}R, \D(K))\to \Ext_R^1 ({^{\phi}}R,\D(M)) \to 0.$$ This implies that
$\Ext_R^1 ({^{\phi}}R,\D(M))=0$. Since $$\Ext_R^1 ({^{\phi}}R,\D(M))\cong \D(\Tor_1 ^R (M, {^{\phi}}R)),$$ we
deduce that $\Tor_1^R (M, {^{\phi}}R)=0$. Therefore ${^{\phi}}R$ is a flat $R$-module, and so by Lemma \ref{2.4}
the ring $R$ is regular and the quotient ring $R/{\phi(\fm)R}$ is Artinian.
\end{proof}

We now present the following corollary, which can be viewed as a dual statement to Kunz's theorem.

\begin{corollary}\label{4.2a} Let $(R,\fm,\Bbbk)$ be a local ring of prime characteristic $p$, and let $f:R \to R$
be the Frobenius map. Let $\mathbb{A}(R)$ denote the category of Artinian $R$-modules. The following are equivalent:
\begin{itemize}
\item[(i)] $R$ is regular.
\item[(ii)] The functor $\Hom_R({^{f}}R,-)$ is exact on $\mathbb{A}(R)$.
\end{itemize}
\end{corollary}

\begin{proof}If $x_1, \ldots, x_{\ell} \in R$ generate $\fm$, then their $p$-th powers generate $f(\fm)R$. Thus,
the ideal $f(\fm)R$ is $\fm$-primary, and so the quotient ring $R/f(\fm)R$ is Artinian. The result then follows
directly from Theorem \ref{4.1a}.
\end{proof}

The following lemma will be crucial for establishing our second characterization of regular rings. It
refines \cite[Korollar 5.5]{HE74}.

\begin{lemma}\label{4.3a}  Let $(R,\fm,\Bbbk)$ be a local ring with a contracting endomorphism $\phi:R\rightarrow R$.
Assume that $R$ is $\F^{\phi}\!$-finite and Cohen-Macaulay. The following are equivalent:
\begin{itemize}
\item[(i)] $R$ is Gorenstein.
\item[(ii)] $R$ possesses a dualizing module $\omega_R$ and for every finitely generated $R$-module $N$ with
$\id_R(N)<\infty$ and for all $e>0$, there exists a functorial isomorphism $\widetilde{\F}^{\phi^{e}}(N)\cong
\F^{\phi^{e}}(N)$.
\item[(iii)] $R$ possesses a dualizing module $\omega_R$ and for every finitely generated $R$-module $N$ with
$\id_R(N)<\infty$,  there exists a functorial isomorphism $\widetilde{\F}^{\phi}(N)\cong \F^{\phi}(N)$.
\end{itemize}	
Moreover, if $R$ is complete, then the above conditions are equivalent to the following two conditions:
\begin{itemize}
\item[(iv)] For every Artinian $R$-module $A$ with $\pd_R(A)<\infty$ and for all $e>0$, there exists a functorial
isomorphism $\widetilde{\F}^{\phi^{e}}(A)\cong \F^{\phi^{e}}(A)$.
\item[(v)] For every Artinian $R$-module $A$ with $\pd_R(A)<\infty$, there exists a functorial isomorphism
$\widetilde{\F}^{\phi}(A)\cong \F^{\phi}(A)$.
\end{itemize}	
\end{lemma}	
	
\begin{proof} (i)$\Rightarrow$(ii) By \cite[Exercise 1.2.26]{BH98}, we have $\depth_{R}({^{{\phi}^e}}\hspace{-0.5mm}R)
=\depth R$. Since $R$ is Gorenstein, its injective dimension is finite. In particular, $R$ is itself a dualizing module
for $R$. By \cite[Theorem 4.15]{PS73}, it follows that $\Ext_R^j({^{{\phi}^e}}\hspace{-0.5mm}R,N)=0$ for every finitely
generated $R$-module with $\id_R(N)<\infty$ and all $e>0$ and $j>0$. Now, the claim follows by an argument analogous to
that used in the proof of \cite[Korollar 5.5]{HE74}.

(ii)$\Rightarrow$(iii) is clear.
	
(iii)$\Rightarrow$(i)  Since the injective dimension of $\omega_R$ is finite, $\widetilde{\F}^{\phi}(\omega_R)\cong
\F^{\phi}(\omega_R)$. Consequently, $\F^{\phi}(\omega_R)\cong \omega_R$ by \cite[Satz 5.12]{HK71}. This implies that
$\omega_R$ is free by \cite[Lemma 2.1(a)]{Rah09}, and therefore $R$ is Gorenstein.

Now, assume that $R$ is complete.

(ii)$\Rightarrow$(iv) Consider an Artinian $R$-module $A$ with $\pd_R(A)<\infty$. Let $$N=\D(A)(=\Hom_R(A,\E_R(\Bbbk))).$$
Then $N$ is a finitely generated $R$-module with $\id_R(N)<\infty$. Hence, we have an isomorphism $\widetilde{\F}^
{\phi^{e}}(N)\cong \F^{\phi^{e}}(N)$. Applying the Matlis duality functor $\D(-)$ to this isomorphism and invoking Corollary
\ref{3.2a}, we obtain $\widetilde{\F}^{\phi^{e}}(\D(N))\cong \F^{\phi^{e}}(\D(N))$. This completes the proof, as Matlis
duality yields $\D(N)\cong A$.

(iv)$\Rightarrow$(v) is clear.

(v)$\Rightarrow$(iii) The proof of this implication is similar to the proof of (ii)$\Rightarrow$(iv). Note that as $R$
is local, every module of finite flat dimension has finite projective dimension by \cite[Proposition 6]{J} and
\cite[Th\'{e}or\`{e}me 3.2.6]{RG}.
\end{proof}

We conclude this section by presenting our second characterization of regular rings, which extends \cite[Korollar 5.6]{HE74}.

\begin{theorem}\label{4.4a} Let $(R,\fm,\Bbbk)$ be a local ring with a contracting endomorphism
$\phi: R\rightarrow R$. Assume that $R$ is $\F^{\phi}\!$-finite. The following statements are equivalent:
\begin{itemize}
\item[(i)] $R$ is regular.
\item[(ii)] For every finitely generated $R$-module $N$ and for all $e>0$, there exists a functorial
isomorphism $\widetilde{\F}^{\phi^{e}}(N)\cong \F^{\phi^{e}}(N)$.
\item[(iii)] For every finitely generated $R$-module $N$, there exists a functorial isomorphism
$\widetilde{\F}^{\phi}(N)\cong \F^{\phi}(N)$.
\end{itemize}	
Moreover, if $R$ is complete, then the above conditions are equivalent to the following two conditions:
\begin{itemize}
\item[(iv)] For every Artinian $R$-module $A$ and for all $e>0$, there exists a functorial isomorphism
$\widetilde{\F}^{\phi^{e}}(A)\cong \F^{\phi^{e}}(A)$.
\item[(v)] For every Artinian $R$-module $A$, there exists a functorial isomorphism $\widetilde{\F}^{\phi}(A)
\cong \F^{\phi}(A)$.
\end{itemize}	
\end{theorem}

\begin{proof} (i)$\Rightarrow$(ii) This follows by the implication (i)$\Rightarrow$(ii) in Lemma \ref{4.3a}.
	
(ii)$\Rightarrow$(iii) This implication is obvious.

(iii)$\Rightarrow$(i) Since the right exact functor $\F^{\phi}$ and the left exact functor $\widetilde{\F}^{\phi}$
are naturally isomorphic on the category of finitely generated $R$-modules, the functor $\F^{\phi}$ must be exact
on this category. This yields that $\Tor_1^R (N,{^{\phi}}R)=0$ for every finitely generated $R$-module $N$, and so
$^{\phi}R$ is a flat $R$-module. Therefore, by Lemma \ref{2.4}, R is regular.

Next, suppose that $R$ is complete. Then by applying the proof of the implications (ii)$\Rightarrow$(iv),
(iv)$\Rightarrow$(v) and (v)$\Rightarrow$(iii) in Lemma \ref{4.3a}, we can see that all five conditions are equivalent.
\end{proof}
		
\section{Local cohomology and the functor $\F^{\phi}$}

In this section, we examine several situations in which the functor $\F^{\phi}$ commutes with local cohomology;
see Lemmas \ref{5.2a}, \ref{3.5a}, and \ref{5.3a}. We also investigate conditions under which the Artinianness
of the local cohomology modules $\H^i_{\fa}(R)$ implies their injectivity; see Theorem \ref{5.4a} and Corollary
\ref{5.5a}. These results extend those of \cite{HS93}, which concern Frobenius endomorphisms, to the more general
setting of locally contracting endomorphisms.

For an $R$-module $M$, recall that a prime ideal $\fp$ of $R$ is called a {\it coassociated prime ideal} of $M$ if
$\fp=\Ann_R(L)$ for some Artinian quotient $L$ of $M$. We denote the set of coassociated prime ideals of $M$ by
$\Coass_R(M)$.

\begin{lemma}\label{5.1aa} Let $\phi:R\to R$ be a ring endomorphism, and $M$ an $R$-module. Assume that $R$ is
$\F^{\phi}\!$-finite. Considering $\widetilde{\F}^{\phi}(M)$ and $\F^{\phi}(M)$ with their natural left $R$-module
structures, we have
\begin{itemize}
\item[(i)]  $\Ass_R(\widetilde{\F}^{\phi}(M))=\Ass_R(M)\cap \Supp_R(^{\phi}R)$.
\item[(ii)] $\Coass_R(\F^{\phi}(M))=\Coass_R(M)\cap \Supp_R(^{\phi}R)$.
\end{itemize}	
\end{lemma}

\begin{proof} By \cite[Exercise 1.2.27]{BH98} and \cite[Theorem 1.21]{Y}, respectively, for any finitely
generated $R$-module $N$, we have $$\Ass_R(\Hom_R(N,M))=\Ass_R(M)\cap \Supp_R(N)$$ and $$\Coass_R(M\otimes_RN)=
\Coass_R(M)\cap \Supp_R(N).$$ The claims follow by taking $N ={}^{\phi}\!R$, which is a finitely generated $R$-module
by the $\F^{\phi}\!$-finiteness assumption.
\end{proof}

Later, in Section 5, we will study the behavior of associated primes under the functor $\F^{\phi}_R(-)$ and of
coassociated primes under the functor $\widetilde{\F}^{\phi}_R(-)$; see Corollaries \ref{3.7a} and \ref{3.8a}.

\begin{lemma}\label{5.1ab} Let $\phi:R\to R$ be a locally contracting endomorphism and $M$ an $R$-module. Assume
that $R$ is $\F^{\phi}\!$-finite. Then $\Supp_R(\F^{\phi}(M))=\Supp_R(M)$. In particular, if $A$ is an Artinian
$R$-module, then the $R$-module $\F^{\phi}(A)$ is also Artinain and has the same support as $A$.
\end{lemma}

\begin{proof} Since $^{\phi}R$ is a finitely generated $R$-module, $\F^{\phi}(A)$ is Artinian for every Artinian
$R$-module $A$. It therefore remains to prove the first assertion.

It is clear that $\Supp_R(\F^{\phi}(M))\subseteq \Supp_R(M)$. For the converse inclusion, pick $\fp \in \Supp_R(M)$.
By Lemma \ref{2.3}(v), we have $\F_R^{\phi}(M)\otimes_RR_{\fp}\cong \F_{R_{\fp}}^{\phi_{\fp}}(M_{\fp})$, so it suffices
to show that $\F_{R_{\fp}}^{\phi_{\fp}}(M_{\fp})\neq 0$. Hence, we may assume that $(R,\fm,\Bbbk)$ is a local ring
admitting a locally contracting endomorphism $\phi:R \to R$ such that $R$ is $\F^{\phi}\!$-finite, and it remains to
show that $\F^{\phi}(M)\neq 0$ for every nonzero $R$-module $M$.

It is well known (and straightforward to verify) that an $R$-module $X$ is nonzero if and only if $\Coass_R(X)\neq
\emptyset$. On the other hand, by the proof of Theorem 3.10 in \cite{EY12}, we have $\Supp_R(^{\phi}R)=\Spec R$.
Therefore, Lemma \ref{5.1aa}(ii) yields that $$\Coass_R(\F^{\phi}(M))=\Coass_R(M)\neq \emptyset,$$ and hence
$\F^{\phi}(M)\neq 0$, which completes the argument.
\end{proof}

The following result extends \cite[Corollary 3.6]{HS93}.

\begin{lemma}\label{5.1a}  Let $R$ be a regular ring with a locally contracting endomorphism $\phi:R\rightarrow R$,
and $A$ an Artinian $R$-module. Assume that $R$ is $\F^{\phi}\!$-finite and $\F^{\phi}(A)\cong A$. Then $A$ is an
injective $R$-module.
\end{lemma}
	
\begin{proof} Assume that $\Supp_RA=\{\fm_1,\dots, \fm_t\}$. By applying \cite[Exercise 8.49(iii)]{S}, it is
straightforward  to verify that $A\cong \bigoplus \limits_{i=1}^t A_{\fm_i}$. For each $1\leq i\leq t$, it is
easy to see that the isomorphism $\F^{\phi}(A)\cong A$ induces an isomorphism $\F^{{\phi}_{\fm_i}}(A_{\fm_i})
\cong A_{\fm_i}$, and that every injective $R_{\fm_i}$-module is also injective as an $R$-module. Thus by localizing
at a maximal ideal, we can assume that $(R,\fm,\Bbbk)$ is a regular local with a locally contracting endomorphism
$\phi:R\rightarrow R$ such that $R$ is $\F^{\phi}\!$-finite.

The local ring $\widehat{R}$ is regular. By Lemma \ref{2.3}(vii), $\widehat{\phi}: \widehat{R}\rightarrow
\widehat{R}$ is contracting, and  by Lemma \ref{2.3}(viii), $\widehat{R}$ is $\F^{\widehat{\phi}}$-finite.  
It is known that any Artinian $R$-module $X$ naturally possesses the structure of an Artinian
$\widehat{R}$-module, and the induced natural map $X\otimes_R\widehat{R}\to X$ is an $\widehat{R}$-isomorphism.
Hence, $A\otimes_R\widehat{R}$ is an Artinian $\widehat{R}$-module and, by Lemma \ref{2.3}(v), we have
$$\F _{\widehat{R}}^{\widehat{\phi}}(A\otimes_R \widehat{R})\cong \F^{\phi}(A)\otimes_R\widehat{R}\cong A\otimes_R
\widehat{R}.$$ By Baer's criterion
for injective modules, a module $X$ over a ring $T$ is injective if and only if $\Ext_T^1(T/\fa,X)=0$ for every
ideal $\fa$ of $T$. On the other hand, for any ideal $\fa$ of $R$, there is a natural $\widehat{R}$-isomorphism $$\Ext_R^1(R/\fa,A)\otimes_R\widehat{R}\cong \Ext_{\widehat{R}}^1(\widehat{R}/\fa\widehat{R},A\otimes_R\widehat{R}),$$
and the functor $(-)\otimes_R\widehat{R}$ is faithful. Hence, if the $\widehat{R}$-module $A\otimes_R\widehat{R}$
is injective, then the $R$-module $A$ is also injective. Thus, we may and do assume that $R$ is complete and $\phi$
is contracting.
	
By Corollary \ref{3.2a}(ii), we have $\D (\F^{\phi}(A))\cong \widetilde{\F}^{\phi}(\D(A))$. Since $R$ is
regular and $\D(A)$ is a finitely generated $R$-module, Theorem \ref{4.4a} yields that
$\widetilde{\F}^{\phi}(\D(A))\cong \F^{\phi} (\D(A))$. Thus $\D(A)\cong \F^{\phi}(\D(A))$, and so $\D(A)$
is free by \cite[Lemma 2.1(a)]{Rah09}. Hence, $A$ is injective.
\end{proof}

\begin{example}\label{5.1abc} Let $\phi:R\to R$ be a locally contracting endomorphism such that $R$ is
$\F^{\phi}\!$-finite, and let $A$ be an Artinian $R$-module. By Lemma \ref{5.1ab}, we know that the $R$-module
$\F^{\phi}(A)$ is Artinian and that $\Supp_R(\F^{\phi}(A))=\Supp_R(A)$. One might guess that $\F^{\phi}(A)$
and $A$ are even isomorphic. However, this is not the case. Indeed, assume that $R$ is regular and that $\fm$
is a maximal ideal of $R$, and let $A$ be any proper nonzero submodule of the Artinian $R$-module $\E_R(R/\fm)$.
Then $A$ is not injective, and hence $\F^{\phi}(A)\ncong A$ by Lemma \ref{5.1a}.
\end{example}
		
The next result generalizes \cite[Proposition 2.1(e)]{MA14} to the context of locally contracting endomorphisms.

\begin{lemma}\label{5.2a} Let $(R,\fm,\Bbbk)$ be a local ring, $\phi:R\to R$ a ring endomorphism, and $N$ a finitely
generated $R$-module of dimension $d$. Assume that either $\phi$ is locally contracting or $R$ is $\F^{\phi}\!$-finite.
Then $\F^{\phi}(\H_{\fm}^d(N))\cong \H_{\fm}^d(\F^{\phi}(N))$.	
\end{lemma}

\begin{proof} Let $\fa=\Ann_R(N)$, and let $x_1,\dots, x_d\in \fm$ be such that their images in $R/\fa$ form a system
of parameters. Set $\fb=\langle x_1, \dots, x_d \rangle$. Then the ideal $\fa+\fb$ is $\fm$-primary,  and hence
$$\H_{\fm}^d(N)\cong \H_{(\fa+\fb)}^d(N)\cong \H_{\fb}^d(N).$$ Since $\H_{\fb}^i(-)=0$ for all $i>d$, the functor
$\H_{\fb}^d(-)$ is right exact. Consequently, we obtain the following display of isomorphisms:
$$
\begin{array}{rlllllllllll}
\F^{\phi}(\H_{\fm}^d(N)) &\cong  \F^{\phi}(\H_{\fb}^d(N)) \\
& \cong \H_{\fb}^d(N)\otimes_R {}^{\phi}\!R\\
& \cong \H_{\fb}^d(N\otimes_R {}^{\phi}\!R)\\
& \cong \H_{\fb}^d(\F^{\phi}(N))\\
& \cong \H_{\phi(\fb)R}^d(\F^{\phi}(N)).
\end{array}
$$
The last isomorphism holds by the Independence theorem \cite[Theorem 4.2.1]{BS}. Because $\phi(\fa)R\subseteq
\Ann_R(\F^{\phi}(N))$, we deduce that $$\H_{\phi(\fb)R}^d(\F^{\phi}(N))\cong \H_{(\phi(\fa)R+
\phi(\fb)R)}^d(\F^{\phi}(N))\cong \H_{\phi(\fa+\fb)R}^d(\F^{\phi}(N)).$$
	
Since either $\phi$ is locally contracting or $R$ is $\F^{\phi}\!$-finite, $\phi(\fm)R$ is $\fm$-primary by Lemma
\ref{2.3}(vi) and (x), which implies that the ideal $\phi(\fa+\fb)R$ is also $\fm$-primary. Therefore,
$$\H_{\phi(\fa+\fb)R}^d(\F^{\phi}(N))\cong \H_{\fm}^d(\F^{\phi}(N)),$$ and so $\F^{\phi}(\H_{\fm}^d(N))\cong
\H_{\fm}^d(\F^{\phi}(N))$,  as desired.
\end{proof}

\begin{lemma}\label{3.5a} Let $\fa$ be an ideal of $R$ and let $\phi:R\rightarrow R$ be a ring endomorphism such that
$\Rad(\phi(\fa)R)=\Rad(\fa)$. Let $\ell=\ara(\fa)$. Then $F^{\phi}(\H_{\fa}^{\ell}(R))\cong \H_{\fa}^{\ell}(R)$.
\end{lemma}
	
\begin{proof} There exist $x_1, \ldots, x_{\ell}\in \fa$ such that $\Rad(\fa)=\Rad\left(\langle x_1, \ldots,
x_{\ell}\rangle\right)$, and so $$\H_{\fa}^{\ell}(R)\cong \H_{\langle x_1, \ldots, x_{\ell}\rangle}^{\ell}(R)\cong
\underset{n \in \mathbb{N}}{\varinjlim} \  \frac{R}{\langle{x_1}^n,\ldots , {x_{\ell}}^n\rangle}.$$
Hence, $$\F^{\phi}(\H_{\fa}^{\ell}(R))\cong \underset{n \in \mathbb{N}}{\varinjlim}\ \frac{R}{\langle{\phi(x_1)}^n,
\ldots, {\phi(x_\ell)}^n\rangle} \cong \H_{\langle{\phi(x_1)}, \ldots, {\phi(x_\ell)}\rangle}^{\ell}(R).$$
We have
$$
\begin{array}{rlllllllllll}
\phi(\fa) R &\subseteq  \phi(\Rad(\fa)) R\\
&= \phi(\Rad(\langle x_1, \ldots, x_\ell\rangle)) R\\
&\subseteq \Rad(\langle {\phi(x_1)}, \ldots, {\phi(x_\ell)}\rangle)\\
&\subseteq \Rad(\phi(\fa) R).
\end{array}
$$
By taking radicals in the above expression and using the fact that the radical operation is idempotent, we obtain
$$\Rad\left(\langle {\phi(x_1)}, \ldots, {\phi(x_\ell)}\rangle\right)=\Rad(\phi(\fa)R)=\Rad(\fa).$$
This implies that $$\H_{\langle {\phi(x_1)}, \ldots, {\phi(x_\ell)}\rangle}^{\ell}(R)\cong \H_{\fa}^{\ell}(R),$$ and
the proof is complete.
\end{proof}

Let $\phi:R \to R$ be a ring endomorphism and $\fa$ an ideal of $R$. In view of Lemma \ref{3.5a}, it is natural to
ask under what conditions the equality $\Rad(\phi(\fa)R) = \Rad(\fa)$ holds. We know that this equality holds for
prime ideals when $\phi$ is locally contracting. In the next result, we extend this property to arbitrary ideals
$\fa$ of $R$ in the case where $R$ is regular. This extension will be used in the proof of Lemma \ref{5.3a}.

\begin{lemma}\label{5.3b} Let $R$ be a regular ring with a locally contracting endomorphism $\phi:R\to R$. Then
$\Rad(\phi(\fa)R)=\Rad(\fa)$ for every ideal $\fa$ of $R$.
\end{lemma}

\begin{proof} Let $\fa$ be an ideal of $R$. Clearly, we may and do assume that $\fa$ is proper. To prove the claim,
it suffices to show that the sets of minimal prime ideals over the two ideals $\fa$ and $\phi(\fa)R$ coincide.

By Lemma \ref{2.4aa}, $\phi$ is flat, and thus it is faithfully flat by Lemma \ref{2.3}(vi). Hence, $\phi$
satisfies both the Going-Down and Lying-Over conditions by \cite[Propositions B.1.1 and B.1.2]{SH}. Since $\phi$ is
locally contracting, it can be easily verified that $\phi^{-1}(\fr)=\fr$ for every prime ideal $\fr$ of $R$.

Assume that $\fp$ is a prime ideal of $R$ minimal over $\phi(\fa)R$. By \cite[Lemma B.1.3]{SH}, there
exists a prime ideal $\fq$ of $R$ which is minimal over $\fa$ such that $\phi^{-1}(\fq)=\fp$. But since $\phi^{-1}(\fq)
=\fq$, we have $\fp=\fq$, and so $\fp$ is minimal over $\fa$.

Conversely, assume that $\fp$ is a prime ideal of $R$ minimal over $\fa$. Then $$\phi(\fa)R\subseteq \phi(\fp)R\subseteq
\Rad(\phi(\fp)R)=\fp.$$ Let $\fq$ be a prime ideal of $R$ satisfying $\phi(\fa)R\subseteq \fq\subseteq \fp$. Then $\fa\subseteq
\phi^{-1}(\fq)\subseteq \phi^{-1}(\fp)$, and so $\fq=\fp$. Thus, $\fp$ is minimal over $\phi(\fa) R$.
\end{proof}

Next, we extends \cite[Lemma 1.8]{HS93} to the context of locally contracting endomorphism.

\begin{lemma}\label{5.3a} Let $R$ be a regular ring with a locally contracting endomorphism $\phi:R\to R$ and
$\fa$ an ideal of $R$. Then $\F^{\phi} (\H^i_{\fa} (M)) \cong \H^i_{\fa}(\F^{\phi}(M))$ for every $R$-module $M$
and all $i\geq 0$.
\end{lemma}

\begin{proof} Let $i$ be a non-negative integer. By Lemma \ref{2.4aa}, the ring endomorphism $\phi$ is flat.
Hence, by the Flat Base Change theorem \cite[Theorem 4.3.2]{BS}, we have $$\F^{\phi} (\H^i_{\fa} (M))\cong
\H^i_{\phi(\fa)R}(\F^{\phi}(M)).$$ Lemma \ref{5.3b} implies that $\Rad(\phi({\fa})R)=\Rad(\fa)$, and thus
$$\H^i_{\phi(\fa)R}(\F^{\phi}(M))\cong \H^i_{\fa}(\F^{\phi}(M)).$$ Therefore, $$\F^{\phi} (\H^i_{\fa} (M))
\cong \H^i_{\fa}(\F^{\phi}(M)),$$ as desired.
\end{proof}

Next, we establish a far-reaching generalization of \cite[Corollary 3.8]{HS93}.

\begin{theorem}\label{5.4a} Let $R$ be a regular ring with a locally contracting endomorphism $\phi: R\rightarrow R$.
Let $\fa$ be an ideal of $R$ and $\ell$ be a non-negative integer. Assume that $R$ is $\F^{\phi}\!$-finite. If the $R$-module
$\H^{\ell}_{\fa}(R)$ is Artinian, then it is injective.
\end{theorem}

\begin{proof} By Lemma \ref{5.3a}, we have $\F^{\phi}(\H^{\ell}_{\fa}(R))\cong \H^{\ell}_{\fa}(R)$. Since $\H^{\ell}_{\fa}(R)$
is Artinian, the assertion follows immediately from Lemma \ref{5.1a}.
\end{proof}

We conclude this section by recording the following immediate corollary.
	
\begin{corollary}\label{5.5a}  Let $R$ be a regular ring with a locally contracting endomorphism $\phi: R\rightarrow R$.
Let $\fa$ be an ideal of $R$ with $\dim R/\fa=0$. Assume that $R$ is $\F^{\phi}\!$-finite. Then the $R$-module
$\H^{\ell}_{\fa}(R)$ is injective for all $\ell\geq 0$. In particular, the $R$-module  $\H^{\ell}_{\fm}(R)$ is injective
for every maximal ideal $\fm$ of $R$ and all $\ell\geq 0$.
\end{corollary}
	
\begin{proof} Because $\dim R/\fa=0$, the ideal $\Rad(\fa)$ is the intersection of finitely many maximal ideals of $R$.
Let $$0\lo \E^0\lo \cdots \lo \E^i\lo \cdots$$ be a minimal injective resolutions of $R$. Since, for any prime ideal
$\fp$ of $R$, the Bass numbers of $R$ with respect to $\fp$ are finite, only finitely many copies of $\E_R(R/\fp)$
occur in the decomposition of each $E^i$ into indecomposable injective modules. For any prime ideal $\fp$ of $R$, we
have $$\Gamma_{\fa}(\E_R(R/\fp))=\begin{cases} \E_R(R/\fp) \hspace{.5cm}   \text{if} \hspace{.5cm}   \fa\subseteq \fp \\
0 \hspace{1.7cm}   \text{if} \hspace{.5cm}     \fa\nsubseteqq \fp.
\end{cases}
$$
As the functor $\Gamma_{\fa}$ commutes with direct sums, it follows that $\Gamma_{\fa}(E^i)$ is an Artinian $R$-module
for all $i\geq 0$. Hence, the modules $\H^{\ell}_{\fa}(R)$ are Artinian for all $\ell \geq 0$. Therefore, the claim follows
immediately from Theorem \ref{5.4a}.
\end{proof}

\section{Preservation of injective modules under the functor $\F^{\phi}$}

In this section, we investigate the preservation of injective modules under the functor $\F^{\phi}$. Theorem
\ref{3.6a} is the main result of the section. Our findings extend several results from \cite{MA14}, which concern
the Frobenius endomorphism, to the broader setting of contracting endomorphisms.

Our first result improves \cite[Proposition 3.10]{MA14}.
		
\begin{lemma}\label{3.3a}  Let $(R,\fm,\Bbbk)$ be a local ring with a ring endomorphism $\phi: R\rightarrow R$.
Assume that $R$ is $\F^{\phi}\!$-finite and that either $R$ is complete or $\phi$ is locally contracting. Then the
following are equivalent:
\begin{itemize}
\item[(i)]  $\widetilde{\F}^{\phi} (R) \cong R$.
\item[(ii)] $\F^{\phi}(\E_R(\Bbbk)) \cong \E_R(\Bbbk)$.
\end{itemize}	
\end{lemma}
	
\begin{proof} Assume that $\phi$ is locally contracting. Set $\E=\E_R(\Bbbk)$. As mentioned in the proof of
Lemma \ref{5.1a}, any Artinian $R$-module $A$ naturally possesses the structure of an Artinian $\widehat{R}$-module,
and the induced natural map $A\otimes_R\widehat{R}\to A$ is an $\widehat{R}$-isomorphism. In particular, there is a
natural $\widehat{R}$-isomorphism $$\E_{\widehat{R}}(\widehat{R}/\fm \widehat{R})=\E\cong \E\otimes_R\widehat{R}.$$

By Lemma \ref{5.1ab}, the $R$-module $\F^{\phi}(\E)$ is Artinian, and hence, by Lemma \ref{2.3}(v), we have the
following natural $\widehat{R}$-isomorphisms:
$$\F^{\widehat{\phi}}(\E_{\widehat{R}}(\widehat{R}/\fm \widehat{R}))\cong
\F^{\widehat{\phi}}(\E\otimes_R\widehat{R})\cong \F^{\phi}(\E)\otimes_R \widehat{R}\cong F^{\phi}(\E).$$
Also, we have the following natural $\widehat{R}$-isomorphisms: $$\widehat{R}\otimes_R \widetilde{\F}^{\phi}(R)
\cong \Hom_{\widehat{R}}(\widehat{R}\otimes_R \hspace{0.3mm}^{\phi}R,\widehat{R}\otimes_RR)\cong
\Hom_{\widehat{R}}(^{\widehat{\phi}}\widehat{R},\widehat{R})=\widetilde{\F}^{\widehat{\phi}}(\widehat{R}).$$
The first isomorphism holds due to the $\F^{\phi}\!$-finiteness of $R$, while the second arises from Lemma \ref{2.3}(viii).
By \cite[Exercise 7.5]{E}, for any two finitely generated $R$-modules $M$ and $N$, $M\underset{R}
\cong N$ if and only if $\widehat{R}\otimes_RM\underset{\widehat{R}}\cong \widehat{R}\otimes_RN$. Therefore,
$\widetilde{\F}^{\phi} (R)\cong R$ if and only if $\widetilde{\F}^{\widehat{\phi}}(\widehat{R})\cong \widehat{R}$.
Thus, we may and do assume that $R$ is complete and $\F^{\phi}\!$-finite.

(i)$\Rightarrow$(ii)  By Corollary \ref{3.2a}(ii), we have $$\D(\F^{\phi}(\E))\cong \widetilde{\F}^{\phi}(\D(\E))
\cong \widetilde{\F}^{\phi}(R)\cong R.$$ Since the $R$-module $\F^{\phi}(\E)$ is Artinian, by Matlis duality,
$D(\D(\F^{\phi}(\E)))\cong \F^{\phi}(\E)$. Therefore, $$\F^{\phi}(\E)\cong \D(\D(\F^{\phi}(\E)))\cong \D(R)\cong \E.$$
		
(ii)$\Rightarrow$(i) Applying Corollary \ref{3.2a}(ii) yields the following isomorphisms:
$$R\cong \D(\E) \cong \D(\F^{\phi} (\E)) \cong \widetilde{\F}^{\phi} (\D(\E)) \cong \widetilde{\F}^{\phi} (R).$$
\end{proof}

The preceding result yields the following corollary.
	
\begin{corollary}\label{3.4a} Let $(R,\fm,\Bbbk)$ be a complete local ring with a ring endomorphism $\phi: R\rightarrow R$.
Assume that $R$ is $\F^{\phi}\!$-finite and $\F^{\phi} (\E_R(\Bbbk))\cong \E_R(\Bbbk)$.  Then $\F^{\phi}(I)$ is an injective
$R$-module for every injective $R$-module $I$.
\end{corollary}
	
\begin{proof} Let $G$ be a flat $R$-module. By \cite[Theorem 5.40]{JR79}, $G \cong \underset{i}{\varinjlim}~ L_i$,
where the $L_i$'s are finitely generated free $R$-modules. By Lemma \ref{3.3a}, we deduce that $\widetilde{\F}^{\phi}(L_i)
\cong L_i$ for each $i$. Since $R$ is $\F^{\phi}\!$-finite, the functor $\widetilde{\F}^{\phi}$ commutes with direct
limits, and so $\widetilde{\F}^{\phi} (G) \cong G$.

An $R$-module $X$ is injective if and only if its Matlis dual $\D(X)$ is flat.  Let $I$ be an injective $R$-module.
Then $\D(I)$ is flat. Hence, by Corollary \ref{3.2a}(ii), we have $$D(I)\cong \widetilde{\F}^{\phi} (\D(I))\cong
\D(\F^{\phi}(I)).$$ Therefore, $\D(\F^{\phi}(I))$ is a flat $R$-module, which implies that $\F^{\phi}(I)$ is injective.
\end{proof}

Next, we strengthen Corollary~\ref{3.4a} by replacing the assumption that $R$ is complete with the assumption that
$\phi$ is locally contracting. For this purpose, we first present the following result, which refines
\cite[Proposition~3.5]{MA14} and whose proof adapts the argument given there.

\begin{lemma}\label{3.4ab} Let $\phi: R \to R$ be a locally contracting endomorphism such that $R$ is
$\F^{\phi}\!$-finite. Let $I$ be an injective $R$-module such that $\F^{\phi}(I)$ is injective. Then
$\F^{\phi}(I)\cong I$.
\end{lemma}

\begin{proof} By Matlis's theorem, there is a family $\{{\fp}_{\lambda}\}_{\lambda\in \Lambda}$ of prime
ideals of $R$ such that $$I\cong \underset{\lambda\in \Lambda}\oplus \E_R(R/{\fp_{\lambda}}).$$ Since
$\F^{\phi}(I)$ is injective, and the functor $\F^{\phi}$ commutes with direct sums, it follows that each
$\F^{\phi}(\E_R(R/{\fp_{\lambda}}))$ is also injective. Therefore, it suffices to show that if, for some
prime ideal $\fp$ of $R$, the $R$-module $\F^{\phi}(\E_R(R/{\fp}))$ is injective, then
$\F^{\phi}(\E_R(R/{\fp}))\cong \E_R(R/{\fp})$.

Let $\fp$ be a prime ideal of $R$ such that $\F^{\phi}(\E_R(R/{\fp}))$ is injective. Each element of
$\F^{\phi}(\E_R(R/{\fp}))$ is annihilated by some power of $\phi(\fp)R$, and so, by Lemma \ref{2.3}(vi),
is annihilated by some power of $\fp$. On the other hand, for each $x\in R\setminus \fp$, multiplication
by $\phi(x)$ induces an automorphism of $\F^{\phi}(\E_R(R/{\fp}))$. Hence, $\F^{\phi}(\E_R(R/{\fp}))$
naturally carries an $R_{\fp}$-module structure, which implies that $$\F^{\phi}(\E_R(R/{\fp}))\cong
\F^{\phi}(\E_R(R/{\fp}))\otimes_R {R_{\fp}}$$ as $R_{\fp}$-modules. Moreover, we have $$\E_R(R/{\fp})
\cong \E_R(R/{\fp})\otimes_RR_{\fp}\cong \E_{R_{\fp}}(R_{\fp}/{{\fp}R_{\fp}}).$$ Hence, by Lemma
\ref{2.3}(v), we deduce that $\F^{\phi}(\E_R(R/{\fp}))\cong
\F_{R_{\fp}}^{\phi_{\fp}}(\E_{R_{\fp}}(R_{\fp}/{{\fp}R_{\fp}}))$. Therefore, we may assume that
$(R,\fm,\Bbbk)$ is a local ring and that $\phi:R\to R$ is a locally contracting endomorphism such that
$R$ is $\F^{\phi}\!$-finite and $\F^{\phi}(\E_R(\Bbbk))$ is an injective $R$-module.  We aim to prove
that $\F^{\phi}(\E)\cong \E$, where $\E=\E_R(\Bbbk)$. By Lemma \ref{5.1ab}, $\F^{\phi}(\E)$ is a nonzero
Artinian $R$-module. Consequently, since $\F^{\phi}(\E)$ is injective, it must be isomorphic to $\E^{n}$
for some $n>0$. Thus, it remains to show that $n=1$. We may further assume that $R$ is complete and
$\F^{\phi}\!$-finite.

Let $d=\dim R$, and set $\omega_R\cong \D(\H_{\fm}^d(R))(=\Hom_R(\H_{\fm}^d(R),\E))$. Since $R$ is complete,
the $R$-module $\omega_R$ is finitely generated. Hence, there exists a surjection $R^s\to \omega_R$ for
some integer $s>0$, which induces the exact sequence  

\begin{equation}
\H_{\fm}^d(R)^s\to \H_{\fm}^d(\omega_R)\to 0.  \label{10}
\end{equation}

Let $U$ be the intersection of all primary components in a minimal primary decomposition of the zero ideal of $R$
whose corresponding quotient has dimension $d$.  By parts (d) and (f) of \cite[Proposition 2.3]{MA14}, there is an
exact sequence $$0\to R/U \to \D(\H_{\fm}^d(\omega_R)).$$ Dualizing gives a surjection $\H_{\fm}^d(\omega_R)\to K$,
where $K=\Hom_R(R/U,\E)$.  From this and \eqref{10}, we obtain the following exact sequence

\begin{equation}
\H_{\fm}^d(R)^s \to K \to 0.  \label{11}
\end{equation}

Applying the Matlis duality functor $\D$ to the exact sequence $0 \to U \to R \to R/U \to 0$ gives the exact sequence
$ 0 \to K \to \E \to \D(U) \to0.$  Combining this with \eqref{11}, we obtain the exact sequence $$\H_{\fm}^d(R)^s\to
\E \to \D(U)\to 0.$$ Lemma \ref{5.2a} implies that $\F^{\phi}(\H_{\fm}^d(R))\cong \H_{\fm}^d (R)$. Thus, applying the
functor $\F^{\phi^e}$ to the latter sequence and using the fact that $\F^{\phi}(\E)\cong \E^n$, we obtain the exact
sequence $$\H_{\fm}^d(R)^s \to \E^{n^e}\to \F^{\phi^e}(D_R(U)) \to 0.$$  Dualizing once more yields an exact sequence

\begin{equation}
0 \to \D(\F^{\phi^e} (\D(U))) \to R^{n^e} \to (\omega_R)^s. \label{12}
\end{equation}

Let $\fa=\Ann_R(U)$. As, by Lemma \ref{2.3}(xi), $R$ is $\F^{\phi^e}$-finite, we conclude that $ \dim R/\phi^e(\fa)
R\leq \dim R/\fa$. For any $R$-module $X$, it is clear that $\Ann_R(X)=\Ann_R(\D(X))$. So, $$\phi^e(\fa) R\subseteq
\Ann_R(\F^{\phi^e}(\D(U)))=\Ann_R(\D( \F^{\phi^e}(\D(U)))).$$ Hence,
$$
\begin{array}{rlllllllllll}
\dim_R(\D( \F^{\phi^e}(\D(U)))) &=\dim R/\Ann_R(\D( \F^{\phi^e}(\D(U))))\\
&\leq  \dim R/\phi^e(\fa) R\\
&\leq \dim R/\fa \\
& < d.
\end{array}
$$

Let $\fp$ be a prime ideal of R such that $\dim R/{\fp}=d$. Localizing \eqref{12} at $\fp$ gives the exactness of $0\to
R_{\fp}^{n^e} \to (\omega_{R_{\fp}})^s$. If $n > 1$, we obtain a contradiction by comparing lengths, since $e$ can
be arbitrarily large.
\end{proof}

\begin{corollary}\label{3.4ac} Let $(R,\fm,\Bbbk)$ be a local ring with a locally contracting endomorphism
$\phi: R\rightarrow R$. Assume that $R$ is $\F^{\phi}\!$-finite and $\F^{\phi}(\E_R(\Bbbk))\cong \E_R(\Bbbk)$.
Then $\F^{\phi}(I)\cong I$ for every injective $R$-module $I$.
\end{corollary}

\begin{proof} Let $I$ be an injective $R$-module. Applying the argument of the proof of Corollary \ref{3.4a}
shows that $\F^{\phi}(I)$ is also an injective $R$-module. By Lemma \ref{3.4ab}, it then follows that
$\F^{\phi}(I)\cong I$.
\end{proof}

The following result improves \cite[Lemma 1.4]{HS93}.
	
\begin{lemma}\label{3.5b} Let $(R,\fm,\Bbbk)$ be a quasi-Gorenstein local ring and $\phi: R\rightarrow R$
a ring endomorphism.
\begin{itemize}
\item[(i)] If the quotient ring $R/{\phi(\fm)R}$ is Artinian, then $F^{\phi}(\E_R(\Bbbk))\cong \E_R(\Bbbk)$.
\item[(ii)] If $R$ is complete and $\F^{\phi}\!$-finite, then $\F^{\phi}(I)$ is an injective $R$-module for every
injective $R$-module $I$.
\end{itemize}
\end{lemma}
	
\begin{proof} (i) Let $d=\dim R$. Then $\ara(\fm)=d$. Since $R$ is quasi-Gorenstein, we have $\E_R(\Bbbk)=\H_{\fm}^d(R)$.
On the other hand, as the quotient ring $R/{\phi(\fm)R}$ is Artinian, it follows that $\Rad(\phi(\fm)R)=\fm$. Thus, the
conclusion follows by Lemma \ref{3.5a}.

(ii) As $R$ is $\F^{\phi}\!$-finite, the quotient ring $R/{\phi(\fm)R}$ is Artinian by Lemma \ref{2.3}(x). Thus the claim
is immediate by (i) and Corollary \ref{3.4a}.
\end{proof}
	
We now present the main result of this section, extending \cite[Proposition 1.5]{HS93}. Here, a quasi-Gorenstein ring
is a ring $R$ such that, for every prime ideal $\frak p$ of $R$, the localization $R_{\frak p}$ is a local
quasi-Gorenstein ring; this definition is well defined by \cite[Corollary 2.4]{A}.

\begin{theorem}\label{3.6a} Let $R$ be a quasi-Gorenstein ring and $\phi:R \to R$ a locally contracting endomorphism.
Then $F^{\phi}(I)\cong I$ for every injective $R$-module $I$.
\end{theorem}
	
\begin{proof} Since, by Matlis' theorem, any injective module is a direct some of indecomposable injective modules and
the functor $\F^{\phi}$ commutes with direct sums, it suffices to show that $\F^{\phi}(\E_R(R/{\fp}))\cong \E_R(R/{\fp})$
for every prime ideal $\fp$ of $R$.
		
Let $\fp$ be a prime ideal of $R$. As we saw in the proof of Lemma \ref{3.4ab}, there is an $R_{\fp}$-isomorphism
$\F^{\phi}(\E_R(R/{\fp}))\cong \F_{R_{\fp}}^{\phi_{\fp}}(\E_{R_{\fp}}(R_{\fp}/{{\fp}R_{\fp}}))$. As the map $\phi_{\fp}:
R_{\fp}\to R_{\fp}$ is a locally contracting endomorphism, Lemma \ref{2.3}(vi)  yields that $\Rad(\phi_{\fp}
(\fp R_{\fp})R_{\fp})=\fp R_{\fp}$, and so the quotient ring $R_{\fp}/{\phi_{\fp}({\fp}R_{\fp}) R_{\fp}}$ is Artinian.
Since the local ring $R_{\fp}$ is quasi-Gorenstein, we have $$\F_{R_{\fp}}^{\phi_{\fp}}
(\E_{R_{\fp}}(R_{\fp}/{{\fp}R_{\fp}}))\cong
\E_{R_{\fp}}(R_{\fp}/{{\fp}R_{\fp}})$$ by Lemma \ref{3.5b}(i), and the result follows by the isomorphism $\E_{R_{\fp}}(R_{\fp}/{{\fp}R_{\fp}})\cong \E_R(R/{\fp})$.
\end{proof}
	
Below, we provide an alternative proof for the implication (i)$\Rightarrow$(iii) in Lemma \ref{2.4aa}. See also
\cite[Corollary 1.6]{HS93}.

\begin{corollary}\label{3.7a} Let $R$ be a regular ring with a locally contracting endomorphism $\phi:R\to R$
and $M$ an $R$-module. Then $\Ass_R(\F^{\phi}(M))=\Ass_R(M)$.
\end{corollary}
	
\begin{proof} Let $\E_R(M)$ denote the injective envelope of $M$. By Lemma \ref{2.4aa}, the functor $\F^{\phi}$
is exact. Applying $\F^{\phi}$ to the natural exact sequence $0\to M\to \E_R(M)$, we obtain the exact sequence
$0\to \F^{\phi}(M)\to \F^{\phi}(\E_R(M)).$ By Theorem \ref{3.6a}, $\F^{\phi}(\E_R(M))\cong \E_R(M)$, which yields
an injective map $\F^{\phi}(M)\to \E_R(M)$. Consequently, we conclude that $$\Ass_R(\F^{\phi}(M))\subseteq
\Ass_R(\E_R(M))=\Ass_R(M).$$

Conversely, let $\fp \in \Ass_R(M)$. Then there is an exact sequence $0 \to R/{\fp}\to M$, and applying the
exact functor $\F^{\phi}$ to this sequence yields the exact sequence $0\to R/{\phi(\fp)R}\to \F^{\phi}(M)$.
Since, by Lemma \ref{2.3}(vi), $\Rad(\phi(\fp)R)=\fp$, it follows that $\fp \in \Ass_R(R/{\phi(\fp)R})$, and
hence $\fp\in \Ass_R(\F^{\phi}(M))$.
\end{proof}
	
Next, we establish a dual result to Corollary \ref{3.7a} for coassociated prime ideals.

\begin{corollary}\label{3.8a} Let $(R,\fm,\Bbbk)$ be a regular local ring with a locally contracting endomorphism
$\phi:R\to R$ and $M$ an $R$-module. Assume that $R$ is $\F^{\phi}\!$-finite. Then $\Coass_R(\widetilde{\F}^{\phi}(M))
=\Coass_R(M)$.
\end{corollary}

\begin{proof} Recall that  $\D(-)=\Hom_R(-,\E_R(R/\fm))$ denote the Matlis duality functor. By \cite[Theorem 1.7]{Y},
for any $R$-module $X$, we have $\Coass_R(X)=\Ass_R(\D(X))$. On the other hand, Corollary \ref{3.2a}(i) implies that
$D(\widetilde{\F}^{\phi}(M))\cong \F^{\phi}(\D(M))$. Thus, by Corollary \ref{3.7a}, we deduce that
\[\begin{array}{rlllllllllll}
\Coass_R(\widetilde{\F}^{\phi}(M)) &=\Ass_R(\D(\widetilde{\F}^{\phi}(M)))\\
&= \Ass_R(\F^{\phi}(\D(M)))\\
&= \Ass_R(\D(M))\\
&= \Coass_R(M).
\end{array}\]
\end{proof}


\end{document}